\documentclass[12pt]{amsart}
\usepackage{amsmath,amscd,amsthm,amsfonts, amssymb,amsxtra, mathrsfs,enumerate}
\usepackage{calrsfs}
\usepackage[us,12hr]{datetime}
\usepackage[all]{xy} \SelectTips{cm}{}
\usepackage{hyperref}
 \hypersetup{colorlinks=true,citecolor=blue}
   \usepackage{graphicx}
   \usepackage[font=scriptsize]{caption}
     \usepackage[margin=1.5in]{geometry}
\usepackage{todonotes}
\usepackage{tikz-cd}

\CDat
\newtheorem{thm}{Theorem}[section]  
\newtheorem*{un-no-thm}{Theorem}
\newtheorem*{un-no-ex}{Example}
\newtheorem*{un-no-defn}{Definition} 
\newtheorem{cor}[thm]{Corollary}     
\newtheorem{lem}[thm]{Lemma}         
\newtheorem{prop}[thm]{Proposition}  
\newtheorem{add}[thm]{Addendum}

\newtheorem{bigthm}{Theorem}
\newtheorem{bigcor}[bigthm]{Corollary}

\theoremstyle{definition}
\newtheorem{defn}[thm]{Definition}   

\theoremstyle{definition}

\theoremstyle{definition}
\theoremstyle{remark}
\newtheorem{rem}[thm]{Remark}

\newtheorem*{out}{Outline}

\newtheorem*{intro-rem}{Remark}
\newtheorem*{intro-rems}{Remarks}

\newtheorem{ex}[thm]{Example}

\DeclareMathOperator*{\hocolim}{hocolim}
\DeclareMathOperator*{\colim}{colim}

\DeclareMathOperator{\PL}{PL}
\DeclareMathOperator{\Or}{\! O}
\DeclareMathOperator{\Gr}{\! G}
\DeclareMathOperator{\SGr}{\! SG}
\DeclareMathOperator{\SOr}{\! SO}
\DeclareMathOperator{\SF}{\! SF}

\DeclareMathOperator{\map}{map}
\DeclareMathOperator{\GL}{GL}

\DeclareMathOperator{\pd}{pd}

\DeclareMathOperator{\cat}{cat}

\begin{document}
\title{On Embeddings and Acyclic Maps}
\author{John R. Klein} 
\address{Wayne State University,
Detroit, MI 48202} 
\email{klein@math.wayne.edu} 

\thanks{}
 \subjclass[2020]{Primary: 58D10, 57R40; Secondary: 57P10, 55P25, 55N25}
\begin{abstract} Given an acyclic map $X\to Y$ of closed manifolds dimension $d$,
we study the relationship between the embeddings
of $Y$ in $S^{n}$ with those of $X$ in $S^{n}$ when $n-d \ge 3$. 

The approach taken here is to first solve the
Poincar\'e duality space variant of the problem. We then apply the surgery machine to obtain
the corresponding results for manifolds.
Thereafter, we focus on the case when $X$ is a smooth homology sphere and deduce results about the homotopy
type of the space of block embeddings of $X$ in a sphere.
\end{abstract}

\maketitle
\setlength{\parindent}{15pt}
\setlength{\parskip}{1pt plus 0pt minus 1pt}
\def\smsh{\wedge}
\def\flush{\flushpar}
\def\dbslash{/\!\! /}
\def\:{\colon\!}
\def\Bbb{\mathbb}
\def\bold{\mathbf}
\def\cal{\mathcal}
\def\End{\text{\rm End}}
\def\stableto{\mapstochar \!\!\to}
\def\lr{{\ell r}}
\def\ad{\text{\rm{ad}}}
\def\tr{\sigma}
\def\Top{\text{\rm Top}}
\def\scr{\mathscr}

\setcounter{tocdepth}{1}
{
  \hypersetup{linkcolor=olive}
  \tableofcontents
}

\addcontentsline{file}{sec_unit}{entry}

\section{Introduction} \label{sec:intro}
This paper stems from the observation that an assortment of golden age results about the existence of embeddings of manifolds in Euclidean space
require a connectivity assumption on the manifold \cite{Haefliger-plongement}, \cite{Haefliger-bulletin-embedding}, \cite{Haef-Hirsch}, \cite{Hudson-PL-embed}. As connectivity is a homotopy property,  one might wonder whether or not a weaker homological condition
is sufficient. In certain circumstances, we will show that the answer is yes, and a description of our results in this area appears
in \S\ref{subsec:exist}. Another undertaking of the paper is
to study the homotopy type of $E^b(\Sigma,S^{n+q})$, the space of block embeddings of a smooth homology $n$-sphere
$\Sigma$ in the sphere $S^{n+q}$; see \S\ref{subsec:homology-sphere-embedding}. In the case  $\Sigma = S^n$, we also provide a complete computation of the rational homotopy groups
of $E^b(S^n,S^{n+q})$ for certain values of $n$ and $q$. The outcome of the computation appears in \S\ref{subsec:computation}.
The results stated in  \S\ref{subsec:exist} and \S\ref{subsec:homology-sphere-embedding}  involve the idea of  {\it acyclic map}.

A map of topological spaces $f\: X\to Y$ is said to be {\it acyclic}  if for every system of local coefficients $\scr E$ on $Y$, the induced homomorphism
\[
f_\ast \: H_\ast(X;f^\ast\!\scr E) \to H_\ast(Y;\scr E)
\]
is an isomorphism.  The prototypical example of an acyclic map is Quillen's plus construction
$X\to X^+$ \cite[\S2]{Wagoner-delooping}.  Acyclic maps were studied in depth by  Hausmann and Husemoller \cite{Hausmann-Husemoller}; see also
the more recent work of Raptis \cite{Raptis}.

If $f\: X \to Y$ is an acyclic map such that $X$ is a $d$-dimensional Poincar\'e duality space in the sense of Wall, then so is $Y$ (see Lemma \ref{lem:target-pd} below). 
 In fact, by a handles version of the plus-construction, if  $X$ has the homotopy type of a closed smooth $d$-manifold,
then $Y$ also has the homotopy type of a closed smooth $d$-manifold, provided that $d \ge 5$ (cf.~\cite[\S5]{Hausmann-homology-surgery}).

In the case when $X$ is a smooth homology $n$-sphere,
 this paper generalizes some of the results of  Hausmann's 1972 Comptes Rendus note 
 \cite{Hausmann_homology-sphere} on the set of isotopy classes embeddings of $X$ in the sphere $S^{n+q}$. 
 However, the proofs of these results never appeared. 
 The current work, in particular, not only rectifies this gap in the literature, but  also proves results about the homotopy type of the block embedding space. Hausmann briefly hints that his results are proved using the embedding theorems of Hudson and Irwin.
By contrast,  our approach is to combine some of the homotopy theory of  Poincar\'e embeddings with the Browder-Casson-Sullivan-Wall surgery 
machine \cite[chap.~11]{Wall_book}.

\subsection{Existence theorems} \label{subsec:exist}
A space $X$ is {\it $r$-acyclic} if
\[
\tilde H_\ast(X;\Bbb Z)  = 0
\]
 in degrees $\ast \le r$, where $\tilde H_\ast$ denotes reduced singular homology. 

Our first main result is about Poincar\'e embeddings; it replaces the 
connectivity condition appearing in \cite[thm.~A]{haef} by the corresponding homological one in the
case when the target is a sphere. For the definition of Poincar\'e embedding, see \S\ref{sec:prelim}.

 \begin{bigthm} \label{bigthm:pd-hom-connected} Let $X$  be a Poincar\'e duality space of dimension $d$. Assume that $X$ is $r$-acyclic, $r\le d-3$.
 Then there is a Poincar\'e embedding of $X$ in $S^{2d-r+1}$.
 \end{bigthm}
 
 The next result  says that if $f\: X \to Y$ is an acyclic map of closed smooth manifolds, then in the metastable range, $X$ embeds in a sphere if 
 $Y$ does. Furthermore, one has a converse when $X$ is orientable.
 
 \begin{bigthm} \label{bigthm:metastable} Let $f\: X\to Y$ be an acyclic map of connected closed smooth $d$-manifolds.
Assume that $3d+ 3\le 2n$. 
Then $X$ smoothly embeds in $S^n$ if $Y$ does.
The converse holds if $X$ is orientable.
\end{bigthm}

The following result is a variant of Theorem \ref{bigthm:pd-hom-connected} for smooth manifolds; note the additional restrictions.

\begin{bigthm}\label{bigthm:hom-connected} Let $X$ a closed smooth $d$-manifold. Assume that $X$ is $r$-acyclic with $d \ge 2r+3$.
Then $X$ smoothly embeds in $S^{2d-r}$.
\end{bigthm}

\begin{ex}[Smooth homology spheres] If $\Sigma^d$ is a smooth homology $n=d$-sphere,
then it is homologically $(d-1)$-connected. Apply Theorem \ref{bigthm:hom-connected} with $r = \lfloor(d-3)/2\rfloor$ to conclude
$\Sigma^d$ smoothly embeds in $S^n$, where $n = \lceil \frac{3}{2}(d+1)\rceil$ (cf.~\cite[thm.~1]{Hausmann_homology-sphere}).
\end{ex}

The next result is a version of Theorem \ref{bigthm:metastable} in the $\PL$ case. In this instance, one does not require a metastable range condition.

\begin{bigthm} \label{bigthm:PL} Let $f\: X\to Y$ be an acyclic map of connected closed $\PL$ $d$-manifolds and assume that
  $n \ge \max(5,d+3)$. Assume there exists a (locally flat) $\PL$ embedding of $Y$ in $S^n$. 
Then there is a $\PL$ embedding of $X$ in $S^n$.
The converse holds if $X$ is orientable.
\end{bigthm}

\begin{ex}[$\PL$ homology spheres] Let $\Sigma^d$ be a $\PL$ homology $d$-sphere. If $d < 3$, then $\Sigma^d$ is $\PL$ homeomorphic to $S^d$ and
therefore $\Sigma^d$ embeds in $S^{d+1}$. 

If $d \ge 3$, then by Theorem \ref{bigthm:PL},
there is a $\PL$ embedding of $\Sigma^d$ in $S^{d+3}$. In fact,   \cite[thm.~1]{Hausmann_homology-sphere} asserts there is a $\PL$ embedding of $\Sigma^d$
 in $S^{d+1}$ if $d \ne 3$.
\end{ex}

\begin{ex}[Connected sums] Suppose that $M$ is a closed connected $\PL$ $d$-manifold and let $\Sigma^d$ is a $\PL$ homology $d$-sphere. 
Let $M \sharp \Sigma$ denote the connected sum. Then the evident map $M\sharp \Sigma^d \to M$ is acyclic.
By Theorem   \ref{bigthm:PL}, if  $n\ge \max(5, d+3)$  and $M$ embeds in $S^n$, so does $M\sharp \Sigma^d$.
\end{ex}

\subsection{Embedding spaces of homology spheres}\label{subsec:homology-sphere-embedding}
Let $\Sigma^n$ be a smooth homology $n$-sphere. For $q \ge 3$, we shall consider the space of {\it block embeddings}
\[
E^b(\Sigma,S^{n+q})\, ,
\]
i.e., the realization of the $\Delta$-set whose $k$-simplices are smooth embeddings of $(k+2)$-ads $\Sigma \times \Delta^k \to S^{n+q} \times \Delta^k$
(cf.~\cite[p.~121]{BLR} or \S\ref{sec:embedding-spaces} below).

We will investigate the homotopy type of $E^b(\Sigma,S^{n+q})$  in terms of more familiar spaces.
Kervaire showed that every smooth homology sphere $\Sigma$ is stably parallelizable \cite[p.~70]{Kervaire_homology-sphere}. This enables one to define a {\it stable normal invariant}
\[
\eta\: \Sigma \to \Gr\!/\Or\, ,
\]
in which the target is the homotopy fiber of $B\Or \to B\Gr$, where the latter is the map from the classifying space of stable
vector bundles to the classifying space of stable spherical fibrations
(cf.~\S\ref{sec:homology-sphere-embedding}). We define $ \ell_q(\eta)$ to be the space of solutions to the lifting problem
\[
\xymatrix{
&   \Gr_q\!/\Or_q \ar[d]\\
\Sigma \ar[r]_(.4)\eta \ar@{..>}[ur]   &\Gr\!/\Or
}
\]
in which $\Gr_q\!/\Or_q$ is the homotopy fiber of the map $B\Or_q\to B\Gr_q$ from the classifying space of   $q$-plane bundles to that of $(q-1)$-spherical fibrations.
Here we have taken the liberty of converting the map $\Gr_q\!/\Or_q \to \Gr\!/\Or$ into a fibration.

\begin{bigthm} \label{bigthm:fiber-sequence} Assume $q \ge 3$ and $n \ge 2$.
Then there is a homotopy fiber sequence 
\begin{equation} \label{eqn:fiber-sequence}
 \ell_q(\eta) \to E^b(\Sigma,S^{n+q}) \to  \Gr_{n+q+1}\! /\Gr_q\, .
\end{equation}
In particular, $E^b(\Sigma,S^{n+q})$ is non-empty if and only if $\ell_q(\eta)$ is non-empty.
\end{bigthm}

We will also show that the fiber $ \ell_q(\eta)$ is $(2q-n-4)$-connected. As the base  $\, \Gr_{n+q+1}\!/\Gr_{q}$ is $(q-1)$-connected, it is in particular connected, and we have the following consequence:

\begin{bigcor} \label{bigcor:cor1} 
The map $E^b(\Sigma,S^{n+q}) \to \Gr_{n+q+1}\!/\Gr_{q}$ is  $(2q-n-3)$-connected.
\end{bigcor}

We also have a complementary result in degrees $\le q-2$. Let $\mathfrak F_q$ denote the homotopy fiber of the map $\Gr_q\!/\Or_q\to \Gr\!/\Or$.

\begin{bigthm} \label{bigthm:homotopy-grps}
Assume that $q \ge 3$ and $n\ge 2$.
If $\ell_q(\eta)$ is non-empty, 
then there is an isomorphism
\[
\pi_j(E^b(\Sigma,S^{n+q})) \cong \pi_{j+n}(\mathfrak F_q)
\]
when $j \le q-2$.
\end{bigthm}

\begin{rem} The group $\pi_{j+n}(\mathfrak F_q)$ does not depend on $\Sigma$.
Consequently,
\[
\pi_j(E^b(\Sigma,S^{n+q})) \cong \pi_j(E^b(S^n,S^{n+q}))
\]
assuming $E^b(\Sigma,S^{n+q})$ is non-empty and $j \le q-2$. Note that the latter inequality is automatically satisfied when $j=0,1$.
The $j=0$ case appears in {\cite[thm.~2]{Hausmann_homology-sphere}}.
\end{rem}

The  results above also have implications for the space of (unblocked) smooth embeddings \[
E(\Sigma,S^{n+q})\, .
\] We topologize the latter
as the geometric realization of the simplicial set whose $k$-simplices are smooth embeddings
$\Sigma {\times} \Delta^k \to S^{n+k} {\times} \Delta^k$ which commute with projection to $\Delta^k$.
The inclusion map $E(\Sigma,S^{n+q}) \to E^b(\Sigma,S^{n+q})$  is 
$(q-2)$-connected (cf.~\cite[\S7.1]{GK-smooth}). Combining the latter with Theorem \ref{bigthm:homotopy-grps}, we have the following result.

\begin{bigcor} \label{bigcor:homotopy-grps}
Assume that $q \ge 3$ and $n\ge 2$.
If $\ell_q(\eta)$ is non-empty, then
 there is an isomorphism
\[
\pi_j(E(\Sigma,S^{n+q})) \cong \pi_{j+n}(\mathfrak F_q)
\]
when $j\le q-3$.
\end{bigcor}




\subsection{A computation of the rational homotopy} \label{subsec:computation} If we focus on the case when $\Sigma = S^n$ is the standard sphere,
then with respect to suitable assumptions on $n$ and $q$,
the fiber sequence \eqref{eqn:fiber-sequence} enables one to completely determine the rational homotopy groups of the block embedding space $E^b(S^n,S^{n+q})$.
The standard inclusion $S^n \subset S^{n+q}$ defines a basepoint for $E^b(S^n,S^{n+q})$.
For positive integers  $q$ and $n$, let 
\[
\scr J_{q,n}
\]
be the set of positive integers $j$ that satisfy any of the following constraints:
\begin{enumerate}[(i).]
\item $j = q$, or \\
\item $j = n+q$, or \\
\item $j \equiv 2 \!\! \mod 4$ and  $j  \ge 2q-n-1$, or \\
\item $j \equiv 3 \!\! \mod 4$ and $j  \ge 2q-1$.
\end{enumerate}

\begin{bigthm} \label{bigthm:computation} Assume that $q$ is even, $n\ge 2$ is odd, 
and $q \ge 2n+2$. Then
\[
\pi_j(E^b(S^n,S^{n+q}))\otimes \Bbb Q \,\, \cong \,\, 
\begin{cases}
\Bbb Q, & j \in \scr J_{q,n};\\
0, & \text{\rm otherwise.}
\end{cases}
\]
\end{bigthm}

\begin{rem} In fact, for any choice of $q\ge 3$ and $n \ge 2$, it is not hard to show that the vector space $\pi_j(E^b(S^n,S^{n+q}))\otimes \Bbb Q$ has dimension at most two for all $j\ge 0$.
\end{rem}

\begin{out}  Section \ref{sec:prelim} consists of  preliminary material on acyclic maps, Poincar\'e spaces,
Poincar\'e embeddings, classifying spaces, and lifting problems. 
In section \ref{sec:induced}, we consider an acyclic map of Poincar\'e spaces $X\to Y$
and relate the Poincar\'e embeddings of $Y$ in a sphere with those of $X$. Section \ref{sec:a-d} contains the proofs of Theorems
\ref{bigthm:pd-hom-connected}-\ref{bigthm:PL}. The material of section \ref{sec:embedding-spaces} topologizes the space of Poincar\'e embeddings and
states a version of the  Browder-Casson-Sullivan-Wall theorem. Section \ref{sec:homology-sphere-embedding} contains
the proof of Theorem \ref{bigthm:fiber-sequence}  and Theorem \ref{bigthm:homotopy-grps}. Section \ref{sec:computation}
is about the proof of Theorem \ref{bigthm:computation}. In appendix \ref{sec:app}, we give two other characterizations of acyclic maps
which may be of general interest, although they are not used in the paper.
\end{out}

\section{Preliminaries}\label{sec:prelim} 

We work within the model category $\Top$ of compactly generated spaces \cite[\S2.4]{Hovey_model}.
The fibrations/weak equivalences are the Serre fibrations/weak homotopy equivalences. Every space is fibrant.
We also typically assume that spaces are cofibrant. To arrange this, we apply cofibrant approximation wherever necessary.

\subsection{Acyclic maps}  We  give some examples and recall some properties of acyclic maps.

\begin{ex} Let $f\: X \to Y$ be an acyclic map with $Y$ is $1$-connected. Then by the universal coefficient theorem, 
$f\: X\to Y$ is acyclic if and only if $f_\ast\: H_\ast(X;\Bbb Z) \to H_\ast(Y;\Bbb Z)$ is an isomorphism.
In particular, If $Y = \ast$ is a point, then $f\: X\to \ast$ is acyclic if and only if $X$ is an acyclic space.
\end{ex}

\begin{ex}[The plus construction] Suppose that $X$ is a connected space. Then Quillen's plus construction
$X\to X^+$ is an acyclic map.

As a special case, note that for $X$ a homology $n$-sphere, we have $X^+ \simeq S^n$.
\end{ex}

\begin{lem}[{\cite[prop.~2.2] {Hausmann-Husemoller}}]\label{lem:hereditary} Assume
\[
\xymatrix{
X \ar[r]^f \ar[d] &Y \ar[d]\\
A \ar[r]_g & B
}\]
is a commutative and homotopy Cartesian and $g\: A\to B$ is an acyclic map. Then $f\: X\to Y$ is an acyclic map.
\end{lem}

\begin{lem}[{\cite[defn./prop.~1.2]{Hausmann-Husemoller}}] A fibration $p\: E \to B$ is an acyclic map if an only if every fiber $p^{-1}(b)$ is an acyclic space.
\end{lem}

\begin{rem}
For several  other characterizations of acyclic maps, we refer the reader to  \cite{Hausmann-Husemoller}, \cite{Raptis}, and 
Appendix \ref{sec:app}.
\end{rem}


\subsection{Poincar\'e duality spaces} Recall that an object $X \in \Top$ is {\it Poincar\'e duality space} of (formal) dimension $d$ if
there exists a pair
\[
(\cal L,[X])
\]
in which $\cal L$ is a local coefficient system of abelian groups that is locally isomorphic to 
$\Bbb Z$ (i.e., the {\it orientation sheaf}) and 
\[
[X] \in H_d(X;\cal L)
\]
is a {\it fundamental class} such that for all local coefficient systems $\cal B$, the cap product homomorphism
\[
\cap [X] \: H^\ast(X;\cal B) \to H_{d-\ast}(X;\cal L\otimes\cal B)
\]
is an isomorphism all degrees (cf.~\cite{Wall_Poincare}, \cite{Klein_Poincare}).  
The Poincar\'e
spaces considered in this paper are assumed to be finitely dominated and cofibrant.
Note that if the pair $(\cal L, [X])$ exists it is determined up to unique isomorphism by the above property. 
Also note that closed manifolds are homotopy finite Poincar\'e duality spaces.

\begin{defn}
A map $f\: X\to Y$ of Poincar\'e spaces of the same dimension is {\it orientable} if it preserves orientation sheaves, i.e., if $L$ is an orientation
sheaf for $Y$, then  $f^\ast L$ is an orientation sheaf for $X$.
\end{defn}

\begin{lem}\label{lem:target-pd} Assume $X$ is a Poincar\'e space of dimension $d$. Let $f\: X \to Y$ be an acyclic map.
Then $Y$ is also a Poincar\'e space of dimension $d$.
Furthermore, $f$ is orientable.
\end{lem}

\begin{proof} Up to isomorphism, the choice of orientation sheaf for $X$ corresponds to an element of $H^1(X;\Bbb Z/2)$. 
As $f^*\: H^1(Y;\Bbb Z/2) \to H^1(Y;\Bbb Z/2)$ is an isomorphism, 
there is a rank one local system $L$ on $Y$ such that $f^\ast L$ is an orientation sheaf for $X$.
Then  we may equip $Y$ with orientation sheaf $L$ and fundamental
class $[Y]$ corresponding to $[X] \in H_d(X;f^\ast L) \cong H_d(Y;L)$. Duality for $Y$ is then a consequence
of duality for $X$.
\end{proof}

\begin{ex} The converse to Lemma \ref{lem:target-pd} is not true in general.
By \cite[prop.~4.7]{Klein-Qin-Su}, if $X$ is the punctured Poincar\'e homology 3-sphere, then
the map $X\to \ast$ is acyclic, but $X$ is not a Poincar\'e duality space in the sense of Wall.
However, $X$ is a Poincar\'e duality space (of dimension zero)  in the weaker sense of Browder \cite{Browder-duality}
since duality holds with $\Bbb Z$-coefficients.
\end{ex}

It is worth singling out a special case of Lemma \ref{lem:target-pd}.

\begin{cor} \label{cor:target-pd} Assume $X$ is a Poincar\'e space of dimension $d$. Let $f\: X \to Y$ be an acyclic map
such that the orientation sheaf for $X$ is the constant sheaf $\Bbb Z$.
Then $Y$ is also a Poincar\'e space of dimension $d$ with orientation sheaf $\Bbb Z$.
Furthermore, $f$ is orientable.
\end{cor}

\subsection{The Spivak normal fibration} Suppose that $X$ is a space and 
$\xi\: E\to X$  is a $(k-1)$-spherical fibration. 
Let $D(\xi)$ be the mapping cylinder of $\xi$,  and $S(\xi) = E$. The {\it Thom space} is 
the quotient space
\[
X^\xi = D(\xi)/S(\xi)\, ,
\]
i.e, the mapping cone of $\xi$.

The obstruction to orienting $\xi$, i,.e., the first Stiefel-Whitney class, defines a rank one local system $L$ on $X$, and one  has a Thom class
$U \in H^k(D(\xi),S(\xi);L)$ inducing the Thom isomorphism
\[
\cap U\: \tilde H_{\ast+k}(X^\xi;\Bbb Z) @> \cong >> H_\ast(X;L) 
\]
(where we have conveniently identified the homology of $D(\xi)$ with the homology of $X$).
 A {\it Spivak normal fibration} for $X$ is a pair
\[
(\xi,\alpha)
\]
in which $\xi\: E \to X$ is a $(k-1)$-spherical fibration and  $\alpha\: S^{d+k}\to X^\xi$ is a map such that if we set
\[
[X] := U\cap \alpha_{\ast} [S^{d+k}] \in H_d(X;L)\, ,
\]
then the data $(L,[X])$ equip $X$ with the structure of a Poincar\'e duality space of dimension $d$. 

Conversely, 
if $X$ admits the structure of a Poincar\'e duality space with data $(L,[X])$, then a Spivak fibration for $X$ exists
inducing $(L,[X])$ provided that
$k$ is sufficiently large. As $k$ tends to infinity, a Spivak fibration is unique up to contractible choice \cite{Spivak}, \cite{klein_dualizing}.
In what follows, when referring to {\it the} Spivak fibration, we implicitly assume that 
$k$ is large, i.e., the spherical fibration is stable.

\subsection{Poincar\'e embeddings in a sphere}
Let $X$ be Poincar\'e duality space of dimension $d$. Let $n \ge d+3$ be an integer. 
A {\it $k$-spherical fibration up to homotopy} over $X$ is a map $E\to X$ whose homotopy
fibers are homotopy equivalent to $S^k$.

\begin{defn}\label{defn:PD-embedding}  A {\it Poincar\'e embedding} of $X$ in $S^n$ consists of a commutative 
homotopy Cocartesian square of spaces
\begin{equation} \label{eqn:PD-embedding}
\xymatrix{
E\ar[r] \ar[d]_{\xi} & C \ar[d] \\
X \ar[r] & S^n
}
\end{equation}
such that 
\begin{enumerate}[(i).]
\item The map $\xi\:E\to X$ is an $(n-d-1)$-spherical fibration, and
\item The collapse homomorphism $H_n(S^n) \to H_{n}(X^\xi)$ is an isomorphism, where $X^\xi$ is the mapping cone (Thom space) of $\xi$.
\end{enumerate}
Here, the collapse homomorphism is defined by observing that there is a weak equivalence of the mapping cones of the
vertical maps in the square. The map $E\to X$ is {\it normal data} and $C$ is the {\it complement}. 
Note that  $C$ is Spanier-Whitehead $(n-1)$-dual to $X$.
\end{defn}

\begin{rem} The spherical fibration $\xi\: E\to X$ together with the collapse
\[
S^n \to \text{cone}(C \to S^n) \simeq \text{cone}(\xi) = X^\xi
\]
is a Spivak normal fibration for $X$.
\end{rem}

\begin{rem} Wall's definition Poincar\'e embedding \cite[p.~119]{Wall_book} is incomplete as it does not require condition (ii) 
above. To see this, consider the commutative homotopy coCartesian square
\[
\xymatrix{
S^{n-1} \ar[r] \ar[d] & S^{n-1} \vee S^n \ar[d] \\
\ast \ar[r] & S^n
}
\]
in which the top horizontal map is the inclusion of the left wedge summand, and the right vertical map
pinches the summand $S^{n-1}$ to a point.  
Note that
condition (ii) of Definition  \ref{defn:PD-embedding} is violated, since the collapse homomorphism in this case is trivial.

However, it is not reasonable to declare this to be a Poincar\'e embedding of a point in $S^n$, since then the complement
in this case, i.e.,
$S^{n-1} \vee S^n$, is not Spanier-Whitehead dual to a point.

If we correct for this discrepancy by appending our condition (ii) to Wall's definition, there is a yet another distinction between the definitions. 
Wall only defines {\it simple} Poincar\'e embeddings of simple Poincar\'e complexes. However,  the additivity formula for Whitehead torsion
implies that when $d\le n+3$, one may avoid simple homotopy theory:
If $X$ is simple Poincar\'e complex of dimension $d$, then $X$ Poincar\'e embeds in $S^n$ in Wall's sense if and only if $X$ Poincar\'e embeds in $S^n$ our sense.
We leave this as an exercise for the reader.
\end{rem}

We consider next triads of the form $Z_\bullet := (Z,Z_0,Z_1)$ having the property that each map
$Z_i \to Z$ is a weak homotopy equivalence; the spaces $Z_i$ are the faces of the triad.
An example of such a triad is $X\times \Delta^1_{{}^\bullet} := (X\times \Delta^1, X\times \{0\}, X\times \{1\})$.

\begin{defn} 
A
{\it concordance} is
a commutative homotopy coCartesian square of triads
\[
\xymatrix{
E_\bullet \ar[r] \ar[d] & C_\bullet \ar[d]\\
X\times  \Delta^1_{{}^\bullet}  \ar[r] & S^n \times \Delta^1_{{}^\bullet}
}
\]
such that the restriction to each face is a Poincar\'e embedding of $X$ in $S^n$.
One says in this case that the pair of Poincar\'e embeddings are {\it concordant}.
Concordance defines an equivalence relation.
\end{defn}

\begin{rem} In Definition \ref{defn:PD-embedding},  we only required
the normal data to be a spherical fibration {\it up to homotopy}. We leave it as an exercise to the reader to show that
a Poincar\'e embedding of $X$ in $S^n$ is (functorially) concordant to one in which
the normal data is an actual spherical fibration.
\end{rem}

\subsection{Lifting problems}
Suppose one is presented with a lifting problem of the form
\begin{equation} \label{eqn:lifting-solutions}
\xymatrix{
Z \ar@/^/[drr]^f \ar@{..>}[dr]\ar@/_/[ddr]_g \\ 
& A' \ar[r]^{u'} \ar[d]_(.4){v} & B' \ar[d]^(.4){v'}\\
& A \ar[r]_{u} & B
}
\end{equation}
in which in which each of the maps appearing in the displayed square is a fibration. 
Without loss in generality, we also assume that the universal map
\[
w\: A' \to A \times_B B'
\]
is a fibration. Note that the pair $(f,g)$ defines a map $h\: Z \to A \times_B B'$.
Hence, we have a lifting problem
\begin{equation} \label{eqn:lifting-solutions2}
 \xymatrix{
& A' \ar[d]^w \\
Z \ar@{..>}[ur] \ar[r]_(.3)h & A \times_B B',
}
\end{equation}

\begin{defn} Given a lifting problem
\[
 \xymatrix{
& E \ar[d]^p \\
Z \ar@{..>}[ur] \ar[r]_f & B,
}
\]
we let  $\ell(f/p)$ denote the space of all lifts, 
i.e., the space of maps $Z \to E$ such that the composition
\[
Z\to E @> w>>  B
\]
is $f$. 
\end{defn}

Then with respect to diagrams \eqref{eqn:lifting-solutions} and \eqref{eqn:lifting-solutions2} we have the following.

\begin{lem} \label{lem:two-descriptions-of-lifts} The space of lifts $\ell(h/w)$ of diagram \eqref{eqn:lifting-solutions2} coincides with the fiber of fibration
\[
\ell(g /v) \to \ell((u\circ g)/ v')\, .
\]
at the basepoint  $f \in \ell(u\circ g \text{ {\rm along} } v')$. 

It is also the fiber of the fibration
\[
\ell(f/u') \to \ell((v'\circ f )/u)\, .
\]
at the basepoint $g\in  \ell((v'\circ f )/u)$.
\end{lem}

\begin{proof} Consider the commutative square of mapping spaces
\[
\xymatrix{
\map(Z,A') \ar[r] \ar[d] & \map(Z,B')\ar[d]\\
\map(Z,A) \ar[r] & \map(Z,B)\, .
}
\]
and note that spaces of the latter with the exception of  the upper left corner are based by the maps $f$, $g$ and $v'\circ f = u\circ g$.
Each of the maps in this square is a fibration.

By definition, the total fiber of this square is the space $\ell(h/w)$. The proof is completed by the observation that the
total fiber has two alternative descriptions. The first is the first take the map of vertical fibers and then to take the fiber again; this gives the first
assertion of the lemma. The second alternative description
is to take the map of horizontal fibers and then to take the fiber again. The latter yields the second assertion of the lemma.
\end{proof}

\subsection{Classifying spaces}

Let $\scr G_q$ denote the category whose objects are spaces $C$ having the homotopy type of the sphere $S^{q-1}$. A morphism
$C\to C'$ is a weak equivalence. 

Recall that the classifying space $B\scr C$ of a category $\cal C$ is the geometric realization of its nerve.

\begin{lem}[cf.~{\cite[prop.~2.2.5]{Wald_1126}}]The space $B\scr G_q$ is a model for $B\Gr_q$, the classifying space of $(q-1)$-spherical fibrations.
\end{lem}

\begin{rem} The category $\scr G_q$ is not small. Nevertheless, there are various ways to address the difficulty of defining its classifying space $B\scr G_q$ (cf.~\cite[p.~379]{Wald_1126}).
\end{rem}

\subsubsection{Stabilization; quotients} 
For spaces
$A$ and $B$, the {\it  join} is defined as
\[
A\ast B := (c(A) \times B) \cup_{A\times B} (A \times c(B))
\]
where $c(A)$ is the cone on $A$. 

\begin{rem} If $A$ and $B$ are based, then $A\ast B \simeq \Sigma (A\smsh B)$, i.e., the suspension of the smash
product of $A$ and $B$.
In particular, if $A$ is $r$-connected and $B$ is $s$-connected, we infer that $A\ast B$ is $(r+s+2)$-connected.
\end{rem}

For $q<n$, the stabilization map 
\[
B\Gr_q \to B\Gr_{n}
\]
 is induced by the functor
\[
S^{n-q-1}{\ast}{-}\: \scr G_q \to  \scr G_{n}
\]
defined by $C\mapsto S^{n-q-1}\ast C$.

\begin{rem} Any homotopy equivalence $W @> \simeq >> S^{k-1}$ induces
a natural transformation of functors 
\[
W{\ast}{-},S^{k-1}{\ast}{-}\:\scr G_q \to \scr G_{q+k}
\]
which is an object-wise equivalence. It follows that the associated maps of classifying spaces
$B\Gr_q \to B\Gr_{q+k}$ are homotopic.
\end{rem}

The homotopy fiber of the map $B\Gr_q  \to  B\Gr_{n}$ at the basepoint is the (homotopy) quotient 
\[
\Gr_{n}\!/\Gr_q \,. 
\]
Concretely, it may be described as the classifying space of the over category  
\[
(S^{n-q-1}{\ast}{-})_{/S^{n-1}}
\]
whose objects are pairs $(C,h)$, in which $C\in \scr G_q$ is an object
and $h\: S^{n-q-1}{\ast}C @> \simeq >>S^{n-1}$ is a homotopy equivalence. A morphism $(C,h) \to (C',h')$ is a homotopy equivalence $f\: C\to C'$ such that
$h'\circ (f\ast \text{id}_C) = h$.

If we set $\scr G = \colim_q \scr G_q$, then $B\scr G$ is a model for $B\Gr$, the classifying space
of stable spherical fibrations. The homotopy quotient 
\[
\Gr\!/\Gr_q,  
\]
has an obvious categorical description that is left to the reader.

There is also a categorical model for the classifying space $B\Or_q$ of real $q$-plane bundles. However, one must take care to include a simplicial direction
that takes into account the topology of the orthogonal group.
Namely, one defines a simplicial category $\scr O_q$ whose objects in simplicial degree $k$ are $k$-dimensional inner product spaces $V$. A  morphism $V\to W$ consists of a
$k$-parameter family of isometries $\Delta^k \times V \to W$ parametrized by the standard $k$-simplex $\Delta^k$. Then $B\scr O_q$ models $B\Or_q$. 
In order to define  the canonical map
$B\Or_q \to B\Gr_q$ we also blow up $\scr G_q$ to a simplicial category $\hat{\scr G}_q$ in a similar way 
and then one has a pair of functors 
\[
\scr O_q\to \hat {\scr G}_q @< \sim << \scr G_q \, ,
\] 
the first which is defined by sending
an inner product space to its unit sphere and the second is defined
as an inclusion by constant families. The latter functor defines an equivalence on classifying spaces.
Note that the map $B\Or_q \to B\Gr_q$ is compatible with stabilization, so we also have a map 
\[
p\:B\Or \to B\Gr
\] 
from the classifying space of stable vector bundles to the classifying space of stable
spherical fibrations. It is well-known that this map is $2$-connected. Without loss in generality, we will assume that $p$ has been converted to a fibration.
The fiber is the quotient $\Gr\!/\Or$.

\subsection{Acyclic maps and relative lifts}
Fix a commutative diagram
\[
\xymatrix{
X \ar[r]^u \ar[d]_f & E \ar[d]^p \\
Y\ar@{..>}[ur] \ar[r]_v & B
}
\]
in which $p$ is a fibration and $f$ is a cofibration.
The dotted arrow represents lifts of $p$ along $Y$ relative to $X$. 

\begin{lem} \label{lem:acyclic-characterization} With respect to the above assumptions, if $f$ is an acyclic map and $p\: E\to B$ induces an isomorphism on fundamental groups, then the space of lifts is weakly contractible. In particular,
it is non-empty.
\end{lem}

\begin{proof} Consider the commutative square of mapping spaces
\[
\xymatrix{
\map(Y,E) \ar[r]^{f^*} \ar[d]_{p_\ast} & \map(X,E) \ar[d]^{p_\ast} \\
\map(Y,B) \ar[r]_{f^*} & \map(X,B)
}
\]
in which all arrows are fibrations.
By \cite[prop.~3.1]{Hausmann-Husemoller}, the fibers of the arrows labeled $f^*$ (taken at the basepoints defined by $f,u$ and $v$) are contractible.
Note that for the top horizontal
map, we have used the $\pi_1$-condition. Consequently, the total fiber of the square with respect to these basepoints is also
contractible. But the total fiber coincides with the space of lifts. 
\end{proof}

Suppose that $f\: X\to Y$ is an acyclic map of Poincar\'e duality spaces
of dimension $d$. Without loss in generality, by replacing $f$ by its mapping cylinder, 
we may assume $f$ is a cofibration.

By Corollary \ref{cor:induced} below, if  $Y \to B\Gr$ classifies the Spivak fibration of $Y$, then the composition
\[
X@> f >> Y \to BG
\]
classifies the Spivak fibration of $X$. If in addition the Spivak fibration is equipped with a lift $X \to B\Or$, then we have a 
lifting problem
\[
\xymatrix{
X \ar[r] \ar[d]_f & B\Or \ar[d]^p \\
Y\ar@{..>}[ur] \ar[r] & B\Gr
}
\]
By Lemma \ref{lem:acyclic-characterization} we have the following.

\begin{cor} \label{cor:acyclic-characterization} With respect to these assumptions, the space
of relative lifts is weakly contractible. In particular, it is non-empty. 
\end{cor}

\begin{rem} A special case occurs when $X$ is equipped with the structure of a smooth manifold.
In this instance, the corollary shows that the map $f\: X\to Y$ can equipped with the structure of a {\it normal map}
(or {\it surgery problem}) in a preferred way.
\end{rem}

\section{Induced Poincar\'e embeddings} \label{sec:induced} Suppose $f\: X\to Y$ is an acyclic map of Poincar\'e spaces
and
 \begin{equation} \label{eqn:uninduced}
\xymatrix{
E\ar[r] \ar[d]_\xi & C \ar[d] \\
Y \ar[r] & S^n
}
\end{equation}
is a Poincar\'e embedding. Without loss in generality, we may assume that $E\to Y$ is a fibration.
Set $E' = E \times_Y X$ and let $\xi'\: E' \to X$ be the projection. Then one has a commutative square
\begin{equation} \label{eqn:induced}
\xymatrix{
E'\ar[r] \ar[d]_{\xi'} & C \ar[d] \\
X \ar[r] & S^n\, .
}
\end{equation}

\begin{lem} \label{lem:induced} The square \eqref{eqn:induced} is a Poincar\'e embedding of $X$.
\end{lem}

\begin{proof} The map
\[
\hocolim(X @<<< E' @>>> C) \to S^n
\]
is a homology isomorphism by a comparison of the Mayer-Vietoris sequences of the squares \eqref{eqn:induced}  and \eqref{eqn:uninduced}. 
The displayed homotopy colimit is 1-connected by the van Kampen theorem. By the relative Hurewicz theorem, the displayed map is a weak equivalence.
Moreover, the collapse  homomorphism $H_n(S^n) \to H_{n-1}(X^{\xi'})$ is an isomorphism since $X^{\xi'} \to Y^{\xi}$ is an acyclic map. The result follows.
\end{proof}

\begin{cor} \label{cor:induced} The Spivak normal fibration of $X$ is the base change of
the Spivak normal fibration of $Y$ along the map $f\: X\to Y$.
\end{cor}

We next show that lemma has a partial converse. Suppose 
\begin{equation} \label{eqn:PDembX}
\xymatrix{
E\ar[r] \ar[d]_{\xi} & C \ar[d] \\
X \ar[r] & S^n\, .
}
\end{equation}
is a Poincar\'e embedding and $f\: X\to Y$ is an acyclic map. 

\begin{lem}\label{lem:converse}
Assume that 
the composition
\[
\text{ker}(\pi_1(f)) \to \pi_1(X) @> w_1(\xi) >> \Bbb Z/2
\]
is trivial, where $w_1(\xi)$ is the first Stiefel-Whitney class of $\xi$.

Then the given Poincar\'e embedding \eqref{eqn:PDembX} is induced up to concordance by a Poincar\'e embedding of $Y$ in $S^n$.
\end{lem}

\begin{rem} The orientation sheaf of $X$ is constant if and only if $w_1(\xi) = 0$. Consequently, the composition in the statement
of the lemma is trivial in this instance.
\end{rem}

\begin{proof}[Proof of Lemma \ref{lem:converse}]  Without loss in generality, we may assume that $f$ is a cofibration and $\xi$ is a fibration.
We let $\hat \xi\:X \to B\Gr_{n-d}$ denote the classifying map
of the given $(n-d-1)$-spherical fibration  $\xi\: E\to X$.

By \cite[prop.~3.1]{Hausmann-Husemoller}, there is a factorization of $\xi$ of the form
\[
X @> f >> Y @>\hat \xi' >> B\Gr_{n-d}
\]
which is unique up to homotopy.
Consequently, $\xi$ is the base change of $\xi'\: E' \to Y$ along $f$.
A similar argument shows that the map $E\to C$ factors up to homotopy as
\[
E \to E' \to C\, .
\]
By the Mayer-Vietoris sequence, it follows that $\hocolim(Y \leftarrow E' \to C)$ is identified with $S^n$, and the conclusion follows.
\end{proof}

\begin{cor}  Let $X \to Y$ by an acyclic map of Poincar\'e spaces.
Then the Spivak normal fibration of $X$ is  trivializable if and only if the Spivak normal fibration of $Y$ is trivializable.
\end{cor}

\begin{proof} The Spivak normal fibration of a Poincar\'e space $X$ defines a Poincar\'e embedding of it in $S^n$ for $n$ sufficiently large and vice-versa.
If the Spivak fibration of $X$ is trivializable, then its first Stiefel-Whitney class is trivial and it follows from Lemma \ref{lem:converse}
that the Spivak normal fibration of $Y$ is trivializable. The converse is a direct consequence of Corollary \ref{cor:induced}.
\end{proof}

\section{Proof of Theorems A-D}\label{sec:a-d}

 \begin{proof}[Proof of Theorem \ref{bigthm:pd-hom-connected}]  Let $Y$ be the plus construction on $X$. Then $Y$ is an $r$-connected Poincar\'e duality space of
 dimension $d$ by Lemma \ref{lem:target-pd}. If $r\le d-3$, then
  $Y$ Poincar\'e embeds in $S^{2d-r+1}$ by \cite[thm.~A]{haef}.
   The result follows by applying Lemma \ref{lem:induced}.
  \end{proof}
 
 \begin{rem} It is possible that \cite[thm.~A]{haef} can be improved by one dimension; see \cite[p.598]{haef}, \cite[thm.~F]{Klein_compress}.
 \end{rem}
 
\begin{proof}[Proof of Theorem \ref{bigthm:metastable}]
If $n \le 5$, then the inequality $3d+ 3\le 2n$ implies that $d\le 2$. In this instance,
$X$ and $Y$ are surfaces and the result is automatic.

Assume then that $n \ge 6$. Then $3d+3\le 2n$ implies
that $d\le n-3$. If $Y$ smoothly embeds in $S^n$ then it Poincar\'e embeds in $S^n$.
By Lemma \ref{lem:induced}, $X$ Poincar\'e embeds in $S^n$. Applying \cite[cor.~11.3.2]{Wall_book},
we obtain a smooth embedding of $X$ in $S^n$. 

Conversely, if $X$ smoothly embeds in $S^n$ then $X$ Poincar\'e embeds in $S^n$.
Apply Lemma \ref{lem:converse} to obtain a Poincar\'e embedding of $Y$ in $S^n$. Then apply 
 \cite[cor.~11.3.2]{Wall_book}, to obtain a smooth embedding of $Y$ in $S^n$.
\end{proof}

\begin{proof} [Proof of Theorem \ref{bigthm:hom-connected}] Set $n := 2d-r$. Then $d \ge 2r+3$ implies $3d+3 \le 2n$, so we are in the metastable range.
 If $r = 0$, then $X$ smoothly embeds in $S^{2d}$ by Whitney's ``hard'' embedding theorem.

Assume then that $ r\ge 1$.
Let $Y = X^+$ be the plus construction; then $Y$ is $r$-connected.
As $d\ge 2r+3 \ge 5$, we infer that $Y$ has the homotopy type of a smooth $d$-manifold $Y'$ \cite[\S5]{Hausmann-homology-surgery}.
Consequently, $Y'$ smoothly embeds
in $S^{n}$ by \cite[p.~51] {Haefliger-plongement}.
By Theorem \ref{bigthm:metastable}, $X$ smoothly embeds in $S^{n}$.
\end{proof}

\begin{proof}[Proof of Theorem \ref{bigthm:PL}]  The proof is basically the same as that of Theorem \ref{bigthm:metastable}, with the exception 
that one now cites \cite[cor.~11.3.1]{Wall_book}.
\end{proof}

\section{Embedding spaces} \label{sec:embedding-spaces}

\subsection{Ads}
Let $I_k$ denote the poset whose objects are
the set of non-empty subsets of the ordered set $[k]:= \{0< 1 <\dots < k\}$,  where the ordering is defined
by inclusion. 
A {\it $(k+2)$-ad} of spaces is a functor
\[
X_\bullet\: I_{k+1} \to \Top
\]
such that for every inclusion $S\subset T$, the map $X_S\to X_T$ is a cofibration.
In particular, a $(k+2)$-ad defines a punctured $(k+1)$-cube of spaces.

Note that the inclusion $d_i\: [k-1] \to [k]$ which forgets the element $i$ defines
a $(k+1)$-ad $d_iX_\bullet$, the $i$-the face of $X_\bullet$. 
A $(k+2)$-ad $X_\bullet$ is {\it special} if each of the maps $X_i \to X$ is a weak equivalence.
The value of $X_\bullet$ at singletons are the vertices of $X_\bullet$.
A map $X_\bullet \to Y_\bullet$ of $(k+2)$-ads is a natural transformation of functors.

If $Z$ is a space, then $Z \times \Delta^k$ 
is a special $(k+2)$-ad in an obvious way. We denote it by $Z\times \Delta^k_{^\bullet}$.

\subsection{The Poincar\'e embedding space}
If $X$ is a Poincar\'e complex of dimension $n$, and let $q\ge 3$ be an integer,  then the space of Poincar\'e embeddings 
of $X$ in $S^n$
\[
E^{\pd}(X,S^{n+q})
\]
is the geometric realization of the Kan simplicial set in which a $k$-simplex
is a commutative square of special $(k+2)$-ads
\[
\xymatrix{
E_\bullet \ar[r] \ar[d] & C_\bullet \ar[d] \\
X\times \Delta^k_{^\bullet}\ar[r] & S^{n+q}\times \Delta^k_{^\bullet}
}
\] 
in which the restriction to each vertex has the structure of a Poincar\'e embedding
of $X$ in $S^{n+q}$.

\subsubsection{Discussion} Up to homotopy,  there is another description  of $E^{\pd}(X,S^n)$ which we won't employ.
However, it may be worth sketching the construction here because it emphasizes a relationship between embedding spaces and 
(homotopy) automorphisms in the Poincar\'e category which is reminiscent to what one has in the manifold categories. The main disadvantage of the
construction is that it involves choices.

If we fix a Poincar\'e embedding
\[
\hskip .2in \xymatrix{
E  \ar[r] \ar[d] & C \ar[d] \\
X\ar[r] & S^{n+q},
}
\] 
\vskip -.62in \hskip 1.5in $\scr D\quad =$\vskip.4in
\noindent i.e., a basepoint in $E^{\pd}(X,S^{n+q})$, then there is a homotopy fiber sequence
\[
\Gr_X(\scr D) \to \Gr_{n+q+1} \to E^{\pd}(X,S^{n+q})\, ,
\]
in which the fiber $\, \Gr_X(\scr D)$ is the topological monoid of (derived) homotopy automorphisms  of the  square $\scr D$ which restrict to the identity on $X$.
The map $\Gr_{n+q+1} \to E^{\pd}(X,S^{n+q})$ is defined by the orbit 
\[
\Gr_{n+q+1} \to \Gr_{n+q+1}\cdot \scr D 
\]
under the evident action of $\Gr_{n+q+1}$ on $E^{\pd}(X,S^{n+q}) $. Then $\Gr_X(\scr D)$ is the stabilizer of $\scr D$ and we have
\[
\Gr_{n+q+1}\cdot \scr D  \,\,  \simeq \,\, \Gr_{n+q+1}\!/\Gr_X(\scr D) \, .
\]
To obtain $E^{\pd}(X,S^{n+q})$, we use a non-connective delooping of $\, \Gr_X(\scr D)$:
Choose a representative $\scr D$ in  each concordance class $[\scr D]$ of Poincar\'e embedding. Then one has a fiber sequence
\[
 E^{\pd}(X,S^{n+q}) \to \coprod_{[\scr D]} B\Gr_X(\scr D) \to B\Gr_{n+q+1}\, .
 \]
 Hence, $ E^{\pd}(X,S^{n+q}) $ is  the homotopy fiber of a non-connective delooping of the homomorphism
 $\Gr_X(\scr D) \to \Gr_{n+q+1}$.

\subsection{The block embedding space}
If $M$ is a smooth $n$-manifold, then the space of block embeddings of $M$ in $S^{n+q}$
\[
E^b(M,S^{n+q})
\]
 is the Kan $\Delta$-set
whose $k$-simplices are given by (locally flat) $\cat$ embeddings of $(k+2)$-ads
\[
M\times \Delta^k_{^\bullet} \to S^{n+q} \times \Delta^k_{^\bullet}\, .
\]
By modifying the definition slightly to include a choice of tubular neighborhood, one obtains a map
\begin{equation}\label{eqn:cat-to-pd}
E^b(M,S^{n+q})\to E^{\pd}(M,S^{n+q})\, .
\end{equation}


\subsection{The Browder-Casson-Sullivan-Wall Theorem}
Let $\nu\: M \to BO$ classify the stable normal bundle, and let
 \[
\scr I^b_q(M)
 \] denote the space of 
lifts of the stable normal bundle $\nu$ to an $q$-plane bundle, i.e., 
the space of lifts
\[
\xymatrix{
& B\Or_{q} \ar[d]  \\
M \ar[r]_{\nu} \ar@{..>}[ur] & B\Or
}
\]
of the classifying map of $\nu$ (here we have taken the liberty of converting the map $B\Or_{q} \to B\Or$ to a fibration
when making this definition). Then there is an evident map 
\[
E^b(M,S^{n+q}) \to \scr I^b_q(M)\, .
\]
Similarly, let $\scr I^{\pd}_q(M)$ denote the space of lifts
\[
\xymatrix{
& B\Gr_{q} \ar[d]  \\
M \ar[r]\ar@{..>}[ur] & B\Gr
}
\]
of the Spivak normal fibration. Then one has a commutative square
\begin{equation} \label{eqn:BCSW-square}
\xymatrix{
E^b(M,S^{n+q})\ar[r] \ar[d] & E^{\pd}(M,S^{n+q})\ar[d]\\
 \scr I^b_q(M)\ar[r] & \scr I^{\pd}_q(M)\, .
}
\end{equation}

\begin{thm}[Wall {\cite[thm.11.3 rel.]{Wall_book}}] \label{thm:smooth-BCSW}
Assume $q \ge 3$ and $n+q \ge 5$. 
Then the square \eqref{eqn:BCSW-square} is homotopy Cartesian.
\end{thm}

\section{The homology sphere case} \label{sec:homology-sphere-embedding} 
We will now focus on the homotopy Cartesian square \eqref{eqn:BCSW-square}
in the case of a homology sphere. Let $\Sigma$ be a smooth homology $n$-sphere.

A Poincar\'e embedding
\[
\xymatrix{
E \ar[r] \ar[d] & C\ar[d]\\
\Sigma \ar[r] & S^{n+q}
}
\]
has the property that the map $E \to \Sigma \times C$ is a weak equivalence and  defines a homotopy trivialization
of the normal data. In particular, $C$ is weakly equivalent to $S^{q-1}$.
and therefore we also obtain a preferred identification of the join $\Sigma \ast C$ with $S^{n+q}$.
It follows that $C$ represents a point in the homotopy fiber of the map 
$\Sigma{\ast}{-} \: B\Gr_q \to B\Gr_{n+q+1}$
given by $C\mapsto \Sigma \ast C$. On the other hand, the plus construction on $\Sigma$ results in an
an acyclic map 
$\Sigma \to S^n$ and yields a homotopy between the maps $\Sigma {\ast} {-}$ and $S^n{\ast}{-}$. In follows that
we may regard $C$ as a point in the fiber of the map $S^n {\ast} {-}$, i.e., of $\Gr_{n+q+1}\!/\Gr_q$. 
We will show that this recipe defines a homotopy equivalence.

\begin{prop} \label{prop:pd-equiv} 
There is a homotopy equivalence
\[
E^{\pd}(\Sigma,S^{n+q}) \,\, \simeq\,\,  \Gr_{n+q+1}\!/\Gr_q\, .
\]
\end{prop}

\begin{proof}  The proof will require some preparation. By slight abuse of notation, in this proof only,
we let $E^{\pd}(\Sigma,S^{n+q})$ denote the simplicial set (rather than the space) used in defining
the Poincar\'e embedding space.
It will be also be convenient to introduce following somewhat imprecise notation:
Given a Poincar\'e embedding
\[
\xymatrix{
E\ar[r] \ar[d] & C\ar[d] \\
\Sigma \ar[r] & S^{n+q},
}
\]
we regard the vertical maps as understood, and denote the data by $(E\to C)$.
Let $c(\Sigma)$ denote the cone on $\Sigma$. There there is an evident zig-zag of equivalences
\begin{equation}\label{eqn:chain}
(E\to C) @>\sim>> (\Sigma \times C @> p_1 >> C) @<\sim << (\Sigma \times C @>\subset >> c(\Sigma) \times C)
\end{equation}
which defines a pair of 1-simplices in $E^{\pd}(\Sigma,S^{n+q})$, i.e., a pair of concordances.

If $A$ and $B$ are spaces, define the {\it thin join} $A\hat\ast B$ by
\[
A \cup_{A\times B} (c(A) \times B)\, .
\]
Then the evident map $A\ast B \to A \hat\ast B$ is a weak equivalence.

Define a simplicial set \[
\scr G_{n+q+1,q}
\] whose $k$-simplices are pairs $(C_\bullet,h_\bullet)$ 
in which $C_\bullet$ is a special $(k+2)$-ad of spaces whose vertices are weakly equivalent to $S^{q-1}$, and 
\[
h_\bullet\: (\Sigma {\times} \Delta^k_{^\bullet} ) \hat\ast C_\bullet \to S^{n+q}{\times} \Delta^k_{^\bullet}
\]
is a weak equivalence of $(k+2)$-ads. Note that $h_\bullet$  determines a map of $(k+2)$-ads
$ \Sigma {\times} \Delta^k_{^\bullet} \to S^{n+q}  {\times} \Delta^k_{^\bullet} $ (this would not have been the case
had we instead used the join).
We leave it as an elementary exercise to show  that there is a homotopy equivalence on realizations
\[
\scr G_{n+q+1,q}\,  \simeq \, \Gr_{n+q+1}\!/ \Gr_q\, .
\]

We define simplicial map $u\: E^{\pd}(\Sigma,S^{n+q}) \to \scr G_{n+q+1,q}$ by
\[
(E_\bullet \to C_\bullet) \mapsto (C_\bullet,h_\bullet)
\]
in which $h_\bullet$  is the equivalence induced by the $(k+2)$-ad of Poincar\'e embeddings $(E_\bullet \to C_\bullet)$.

Define a simplicial map $v\: \scr G_{n+q+1,q} \to E^{\pd}(\Sigma,S^{n+q})$ by
\[
 (C_\bullet,h_\bullet) \mapsto (\Sigma \times C_\bullet \to c(\Sigma) \times C_\bullet)\, .
 \]
 Then the chain \eqref{eqn:chain} shows that the compositions $u\circ v$ and $v\circ u$ are homtopic to the identity in a preferred way.
 \end{proof}

\begin{rem} Using the construction in \S\ref{sec:induced}, the acyclic map $\Sigma \to S^n$  induces a map $E^{\pd}(S^n,S^{n+q}) \to E^{\pd}(\Sigma,S^{n+q})$ which 
is also a homotopy equivalence.
\end{rem}


\begin{proof}[Proof of Theorem \ref{bigthm:fiber-sequence}] 
With respect to the identification of Proposition \ref{prop:pd-equiv}, the right vertical map 
of \eqref{eqn:BCSW-square} is induced by taking complement data $C$ to the constant lift
\[
\Sigma @>>> B\Gr_q \, ,
\]
having value $C$.

As the square of  \eqref{eqn:BCSW-square}, is homotopy Cartesian, it is enough to show that homotopy fiber of the map
\[
 \scr I^b_q(\Sigma)\to  \scr I^{\pd}_q(\Sigma)
 \]
(at the given basepoint) is homotopy equivalent to  $\ell_q(\eta)$.

The smooth structure on $\Sigma$ gives rise
to the factorization $\Sigma \to B\Or\to B\Gr$  and defines a lift of the (trivial) Spivak fibration to the stable normal bundle. 
Although the stable normal bundle in this instance is trivializable, the trivialization depends on the smooth structure. The lift is encoded
by the  {\it stable normal invariant} $\eta\:\Sigma \to \Gr\!/\Or$ associated with the difference between two null homotopies of the composition
$\Sigma \to B\Gr \to B(\, \Gr\!/\Or)$, where  one null-homotopy is defined by the trivialization of the Spivak fibration and the other
defined by the Kervaire's trivialization of the stable normal bundle
 \cite[p.~109]{Wall_book}.

  By Lemma \ref{lem:two-descriptions-of-lifts}, the homotopy fiber of the map
\[
 \scr I^b_q(\Sigma)\to  \scr I^{\pd}_q(\Sigma)
 \]
 at the basepoint is identified with
 the space of lifts
 \[
 \xymatrix{
 & \Gr_q\!/\Or_q \ar[d]\\
 \Sigma \ar@{..>}[ur] \ar[r]_(.4){\eta}  & \Gr\!/\Or\, ,
 }
 \]
where again, we have taken the liberty here of converting $\, \Gr_q\!/\Or_q \to \Gr\!/\Or$ to a fibration.
But this is precisely the definition of $\ell_q(\eta)$.
\end{proof}


Let
\[
\mathfrak F_q := \text{fiber}(\Gr_q\!/\Or_q \to \Gr\!/\Or)
\] 
be the homotopy fiber of the map $\Gr_q\!/\Or_q \to \Gr\!/\Or$.

\begin{lem} \label{lem:fibFq} If $\ell_q(\eta) \ne \emptyset$, then there is a homotopy fiber sequence
\begin{equation} \label{eqn:fibFq}
\Omega^n \mathfrak F_q \to \ell_q(\eta) \to \mathfrak F_q\, .
\end{equation}
\end{lem}

\begin{rem} The fibration is trivializable when $\eta$ is null homotopic. However, it is not trivializable in general.
\end{rem}

\begin{proof}[Proof of Lemma \ref{lem:fibFq}] By \cite[prop~3.1]{Hausmann-Husemoller}, $\eta$ factors as 
\[
\Sigma \to S^n @> \eta' >> \Gr\!/\Or
\]
and induces a homotopy equivalence $\ell_q(\eta) \simeq \ell_q(\eta')$. We may assume that $\eta'$ is a basepoint  preserving map.
Then evaluation at the basepoint of $S^n$ defines a fibration
\[
\ell_q(\eta';\ast) \to \ell_q(\eta') \to \mathfrak F_q
\]
where  $\ell_q(\eta';\ast)$ is the fiber at $\eta'$ of 
\[
\Omega^n (\Gr_q\!/\Or_q) \to \Omega^n (\Gr\!/\Or) \, .
 \]
 Since the latter is an $n$-fold loop map, its fiber at $\eta'$ and its fiber at the constant $n$-fold loop
 are homotopy equivalent. It follows that $\ell_q(\eta';\ast) \simeq \Omega^n \mathfrak F_q$.
\end{proof}



\begin{proof}[Proof of Corollary \ref{bigcor:cor1}] As $\, \Gr_{n+q+1}\!/\Gr_q$ is connected, by Theorem 
 \ref{bigthm:fiber-sequence} it is enough to show that  $\ell_q(\eta)$  is $(2q-n-4)$-connected.
Note that the  map $\, \Gr_q\!/\Or_q \to \Gr\!/\Or$ is $(2q-3)$-connected (cf.~\cite{Haefliger-Diff-embeddings}, \cite[p.~124]{Wall_book}). 
In particular, $\ell_q(\eta) \ne \emptyset $ if $2q \ge n-3$.  
Using Lemma \ref{lem:fibFq}, it follows that   $\ell_q(\eta)$  is $(2q-n-4)$ connected. \end{proof}

\begin{proof}[Proof of Theorem \ref{bigthm:homotopy-grps}] Observe that $\Gr_{n+q+1}/\Gr_{q}$
 is $(q-1)$-connected. 
 Then using the long exact homotopy sequence
\eqref{eqn:fiber-sequence}  we deduce that
 \[
 \pi_j(E^b(\Sigma,S^{n+q})) \cong \pi_j(\ell_q(\eta))\, 
 \]
 for $j\le q-2$.
 Since $\mathfrak F_q$ is $(2q-4)$-connected,
 the fiber sequence \eqref{eqn:fibFq} implies that
 \[
 \pi_j(\ell_q(\eta)) \cong \pi_{k+n}(\mathfrak F_q)
 \]
 for $j \le 2q-5$. The result follows since $q-2 \le 2q-5$ is equivalent to $q\ge 3$.
\end{proof}


\section{Proof of Theorem \ref{bigthm:computation}} \label{sec:computation}

The proof of Theorem \ref{bigthm:computation} will depend on several lemmas.

\begin{lem}  \label{lem:relativeG} If $q$ is even and $n$ is odd, then
\[
\pi_j(\Gr_{n+q+1}\! /\Gr_q) \otimes \Bbb Q \cong
\begin{cases}
\Bbb Q & \text{ \rm if } j = q, n+q,\\
0 & \text{\rm otherwise.}  
\end{cases}
\]
\end{lem}

\begin{proof} Let $\SGr_q$ denote topological monoid of degree one self-maps of $S^{q-1}$. Then
$B\SGr_q$ is the classifying space of oriented $(q-1)$-spherical fibrations and
\[
\SGr_{n+q+1}\! /\SGr_q \, \, \simeq\,\,  \Gr_{n+q+1}\!/\Gr_q \, .
\]
There is a fibration sequence
\[
\SF_{q-1} \to \SGr_q \to S^{q-1}\, ,
\]
in which $\SF_{q-1}$ denotes the topological monoid of based, degree one self-maps of $S^{q-1}$, and
the map $\SGr_q \to S^{q-1}$ evaluates a self-map at the basepoint.

Since $q-1$ is odd, $S^{q-1}$ has the rational homotopy type of the Eilenberg~MacLane space $K(\Bbb Q,q-1)$.
It follows that $\, \SF_{q-1}$ is rationally contractible, and we have  a rational equivalence $B\SGr_q \simeq_{\Bbb Q} K(\Bbb Q,q)$. Since $n+q$ is odd,
a similar argument shows $\, \SGr_{n+q+1} \simeq_{\Bbb Q} K(\Bbb Q,n+q)$. 

The result now follows these observations combined with the long exact sequence of
rational homotopy groups associated with the homotopy fiber sequence
\[
\SGr_{n+q+1}\to  \SGr_{n+q+1}\!/\SGr_q \to B\SGr_q\, . \qedhere
\] 
\end{proof}

\begin{lem} \label{lem:gqmodoq}  If $q$ is even, then is a rational homotopy equivalence
\[
\Gr_q\!/\Or_q \simeq \prod_{k=1}^{q/2-1} K(\Bbb Q,4k)\, .
\]
\end{lem}

\begin{proof} According to \cite[thm.~15.9]{Milnor-Stasheff}, the map
\[
B\SOr_q\,\,  @> (e_q , \prod_{k=1}^{q/2-1}  p_k)>> K(\Bbb Q,q) \times \prod_{k=1}^{q/2-1} K(\Bbb Q,4k)\, ,
\]
is a rational homotopy equivalence, where $e_q$ is the Euler class and $p_k$ is the $k$-th Pontryagin class.

In the proof of Lemma \ref{lem:relativeG} it was observed that  there is rational homotopy equivalence
$B\SGr_q \simeq_{\Bbb Q} K(\Bbb Q,q)$ when $q$ is even.  With respect to these identifications, the map
\begin{equation} \label{eqn:so-sg}
B\SOr_q \to B\SGr_q 
\end{equation}
corresponds to projection onto the first factor $K(\Bbb Q,q)$, since the Euler class of an oriented spherical fibration is well-defined. Consequently, the rational homotopy type of the homotopy fiber of the map \eqref{eqn:so-sg} is
\[
 \prod_{k=1}^{q/2-1} K(\Bbb Q,4k)\, ,
 \]
and the proof is completed using the homotopy equivalence
$\, \Gr_q\!/\Or_q  \simeq \SGr_q\!/\SOr_q$.
\end{proof} 

Recall that $\mathfrak F_q$ was defined in the introduction as the homotopy fiber of the map
$\Gr_q\! /\Or_q \to \Gr\! /\Or$.

\begin{cor} \label{cor:F-rational} There is a rational homotopy equivalence
\[
\mathfrak F_q \, \, \simeq_{\Bbb Q} \,\, \prod_{k = q/2}^\infty K(\Bbb Q,4k-1)\, .
\]
\end{cor}

\begin{proof}  Analogous to Lemma \ref{lem:gqmodoq}, one has a rational equivalence
\[
\Gr\!/\Or @> \prod_{k=1}^\infty p_k >> \prod_{k=1}^\infty K(\Bbb Q,4k)
\]
The map $\Gr_q\!/\Or_q \to \Gr\! /\Or$ corresponds to the inclusion into the first $(q/2-1)$-factors, so the inclusion
is immediate.
\end{proof}

In the case $\Sigma = S^n$, the stable normal invariant $\eta\:S^n \to \Gr\! /\Or$ is null homotopic since $S^n$ is the boundary of the disk $D^{n+1}$, the latter which is a parallelizable manifold. Consequently,  there is a homotopy equivalence
\[
\ell_q(\eta) \simeq\map(S^n,\mathfrak F_q)\, .
\]
The latter sits in a fibration sequence
\[
\Omega^n \mathfrak F_q\to \map(S^n,\mathfrak F_q) \to \mathfrak F_q
\]
which is rationally homotopically trivial since $\mathfrak F_q$ has the rational homotopy type of 
a loop space.

Then by Corollary \ref{cor:F-rational}, we obtain a rational homotopy equivalence
\[
\ell_q(\eta) \,\, \simeq_{\Bbb Q} \,\, \prod_{k=q/2}^\infty K(\Bbb Q,4k{-}n{-}1) \times \prod_{k=q/2}^\infty K(\Bbb Q,4k{-}1) \, .
\]
Therefore,
\begin{equation}\label{eqn:ell-eta}
\pi_j(\ell_q(\eta))\otimes {\Bbb Q} \cong  \begin{cases}
\Bbb Q & \text{ \rm if } j +n \equiv 3\! \!\! \mod 4 \text{ \rm and } j \ge 2q-n-1, \text{ \rm or } \\ 
&   \text{ \rm if } j \equiv 3\!\!\! \mod 4 \text{ \rm and } j \ge 2q -1, \\
0 & \text{ \rm otherwise}.
\end{cases}
\end{equation}

\begin{proof}[Proof of Theorem \ref{bigthm:computation}]
By the  fibration sequence \eqref{eqn:fiber-sequence}, there is a long exact sequence
\[ \small
\xymatrix{
\cdots \ar[r] &
\pi_{j+1}(\Gr_{n+q+1}\!/\Gr_q)\otimes \Bbb Q \ar[r]^(.55){\partial_{j+1}} & \pi_j(\ell_q(\eta)) \otimes \Bbb Q) \ar[r] & 
\pi_j(E^b(S^n,S^{n+q})) \otimes \Bbb Q\ar[dll]  \\
  & \pi_{j}(\Gr_{n+q+1}\!/\Gr_q)\otimes \Bbb Q \ar[r]_(.55){\partial_j} &  \pi_{j-1}(\ell_q(\eta)) \otimes \Bbb Q \ar[r] & \cdots  \\
 }
\]
By Lemma \ref{lem:gqmodoq}, \eqref{eqn:ell-eta},
and the constraints on $n$ and $q$,  it is readily checked that the boundary homomorphisms $\partial_{j+1}$ and $\partial_j$ are trivial. Moreover, the pair of vector spaces $ \pi_j(\ell_q(\eta)) \otimes \Bbb Q), \pi_{j}(\Gr_{n+q+1}\!/\Gr_q)\otimes \Bbb Q$ cannot be simultaneously non-trivial, and when one of these vector spaces is non-trivial, it has dimension one. 
 It is also clear that $\pi_j(\ell_q(\eta)) \otimes \Bbb Q), \pi_{j}(\Gr_{n+q+1}\!/\Gr_q)\otimes \Bbb Q$ are simultaneously trivial if and only
$j \notin \scr J_{n,q}$.
\end{proof}

\appendix

\section{Further characterizations of acyclic maps} \label{sec:app}

Here we  provide two other characterisations of acyclic maps. 

\subsection{The fiberwise suspension characterization} Let $\Top_Y$ denote the category of spaces over  a fixed space $Y$.
Then one has the {\it unreduced fiberwise suspension}.
\[
S_Y X := (Y \times 0) \cup X\times [0,1] \cup (Y \times 1)  \in \Top_Y
\]

\begin{lem} $f\: X\to Y$ is acyclic if and only if $S_YX$ 
is weakly equivalent to the terminal object, i.e., the map $S_Y X \to Y$ is a weak homotopy equivalence of spaces.
\end{lem}

\begin{proof} This is a reformulation of Raptis' \cite[thm.~2.1]{Raptis} ]
\end{proof}

\subsection{The parametrized spectra characterization}
A  parametrized spectrum $\scr E$ over $Y$ gives rise to a cohomology theory on $\Top_Y$. If $X\in \Top_Y$ is
a object, then
\[
H^\bullet(X;\scr E)
\]
is the spectrum of sections of $\scr E$ taken along the  structure map $X\to Y$. The homotopy groups
of this spectrum define the cohomology of $X$ with coefficients in $\scr E$.

\begin{prop} \label{prop:acyclic-characterization}
A map $X\to Y$ of connected spaces acyclic if and only if the induced map of spectra
\[
H^\bullet(Y;\scr E) \to H^\bullet(X;\scr E)
\]
is a weak equivalence for all parametrized spectra $\scr E$.
\end{prop}

\begin{proof}  Assume $X\to Y$ is acyclic and choose a basepoint in $X$; then $Y$ inherits a basepoint and $f$ is a
based map. 
The reduced cohomology is given by
\[
\tilde H^\bullet(X;\scr E) :=  \text{fiber } (\tilde H^\bullet(X;\scr E) \to  \tilde H^\bullet(\ast;\scr E))\, .
\]
Let $F$ denote a homotopy fiber of the map $X\to Y$ and let $X\cup_F CF$ denote the mapping cone
of the map $F\to X$. Then the map 
\[
X \cup_F CF \to Y
\]
is a weak equivalence  \cite[thm.~2.5]{Hausmann-Husemoller}. Consequently, there is fiber sequence of spectra
\[
\tilde H^\bullet(Y;\scr E) \to \tilde H^\bullet(X;\scr E) @>>> \tilde H^\bullet(F;\scr E)\, .
\]
It is enough to show that $\tilde H^\bullet(F;\scr E)$ is weakly contractible.
As the map $F\to Y$ is null homotopic, the base change of $\scr E$ to $F$ is of the form $F\times \scr W$, for some ordinary
spectrum $W$. Consequently, 
\[
\tilde H^\bullet(F;\scr E) \simeq \map_\ast(F,\scr W) \simeq  \map_\ast(\Sigma F,\Sigma \scr W)
\]
and since $\Sigma F$ is weakly contractible, it follows that $\tilde H^\bullet(F;\scr E)$ is weakly contractible.

Conversely, if $
H^\bullet(Y;\scr E) \to H^\bullet(X;\scr E)
$ is a weak equivalence for all $
\scr E$ is follows that it is a weak equivalence for all $\scr E$ whose fibers are Eilenberg-Mac~Lane spectra. But the latter
correspond to the local coefficient systems on $Y$ and the cohomology of $Y$ with coefficients in such $\scr E$
are then the cohomology with twisted coefficients.
\end{proof}





\end{document}